\documentclass{article}%
\usepackage{graphicx}
\usepackage{amsmath}
\usepackage{amsfonts}
\usepackage{amssymb}
\usepackage{setspace}
\usepackage[hmargin=3.5cm,vmargin=3.5cm]{geometry}%
\providecommand{\U}[1]{\protect\rule{.1in}{.1in}}
\providecommand{\U}[1]{\protect\rule{.1in}{.1in}}
\providecommand{\U}[1]{\protect\rule{.1in}{.1in}}
\providecommand{\U}[1]{\protect\rule{.1in}{.1in}}
\providecommand{\U}[1]{\protect\rule{.1in}{.1in}}
\providecommand{\U}[1]{\protect\rule{.1in}{.1in}}
\providecommand{\U}[1]{\protect\rule{.1in}{.1in}}
\providecommand{\U}[1]{\protect\rule{.1in}{.1in}}
\providecommand{\U}[1]{\protect\rule{.1in}{.1in}}
\providecommand{\U}[1]{\protect\rule{.1in}{.1in}}
\providecommand{\U}[1]{\protect\rule{.1in}{.1in}}
\providecommand{\U}[1]{\protect\rule{.1in}{.1in}}
\providecommand{\U}[1]{\protect\rule{.1in}{.1in}}
\providecommand{\U}[1]{\protect\rule{.1in}{.1in}}
\providecommand{\U}[1]{\protect\rule{.1in}{.1in}}
\providecommand{\U}[1]{\protect\rule{.1in}{.1in}}
\providecommand{\U}[1]{\protect\rule{.1in}{.1in}}
\providecommand{\U}[1]{\protect\rule{.1in}{.1in}}
\providecommand{\U}[1]{\protect\rule{.1in}{.1in}}
\providecommand{\U}[1]{\protect\rule{.1in}{.1in}}
\providecommand{\U}[1]{\protect\rule{.1in}{.1in}}
\providecommand{\U}[1]{\protect\rule{.1in}{.1in}}
\providecommand{\U}[1]{\protect\rule{.1in}{.1in}}
\providecommand{\U}[1]{\protect\rule{.1in}{.1in}}
\providecommand{\U}[1]{\protect\rule{.1in}{.1in}}
\providecommand{\U}[1]{\protect\rule{.1in}{.1in}}
\providecommand{\U}[1]{\protect\rule{.1in}{.1in}}
\providecommand{\U}[1]{\protect\rule{.1in}{.1in}}
\providecommand{\U}[1]{\protect\rule{.1in}{.1in}}
\providecommand{\U}[1]{\protect\rule{.1in}{.1in}}
\providecommand{\U}[1]{\protect\rule{.1in}{.1in}}
\providecommand{\U}[1]{\protect\rule{.1in}{.1in}}
\providecommand{\U}[1]{\protect\rule{.1in}{.1in}}
\providecommand{\U}[1]{\protect\rule{.1in}{.1in}}
\providecommand{\U}[1]{\protect\rule{.1in}{.1in}}
\providecommand{\U}[1]{\protect\rule{.1in}{.1in}}
\providecommand{\U}[1]{\protect\rule{.1in}{.1in}}
\providecommand{\U}[1]{\protect\rule{.1in}{.1in}}

\newtheorem{theorem}{Theorem}
{}

\newtheorem{condition}{Condition}

\newtheorem{example}{Example}

\newtheorem{lemma}{Lemma}
{}

\newtheorem{proposition}{Proposition}

\newenvironment{proof}[1][Proof]{\textbf{#1.} }{\ \rule{0.5em}{0.5em}}

\begin{document}

\title{Estimates for Dirichlet Eigenvalues of the Schr\"{o}dinger operator with the
Kronig-Penney Model}
\author{Cemile Nur* and Oktay Veliev**
\and Department of Basic Sciences, Yalova University\\Department of Basic Sciences, Yalova University, Yalova, Turkey \\{\small e-mail: cnur@yalova.edu.tr*} {\small ORCID iD:
https://orcid.org/0000-0001-7375-3474} \\Department of Mechanical Engineering, Dogus University, Istanbul, Turkey \\{\small e-mail: oveliev@dogus.edu.tr**} {\small ORCID iD:
https://orcid.org/0000-0002-0314-9404}}
\date{}
\maketitle

\begin{abstract}
In this paper, we first improve some asymptotic formulas previously obtained
and provide sharp asymptotic formulas explicitly expressed by the potential.
For the potentials of bounded variation, we obtain asymptotic formulas in
which the first and second terms are explicitly determined and separated from
the error terms. In addition, we illustrate these formulas for the
Kronig-Penney potential. We then provide estimates for the small Dirichlet
eigenvalues of the one-dimensional Schr\"{o}dinger operator in the
Kronig-Penney model. We derive several useful equations from certain iteration
formulas for computing these Dirichlet eigenvalues, and prove that all the
eigenvalues can be found by the fixed point iteration. Then, using the Banach
fixed point theorem, we estimate the eigenvalues numerically. Moreover, we
present error estimates and include a numerical example.



\end{abstract}

\section{Introduction and Preliminary Facts}

\label{Sec:1} In the present paper, we consider the operator generated in
$L_{2}[0,\pi]$ by the expression
\begin{equation}
-y^{\prime\prime}+q(x)y \label{eq1}%
\end{equation}
and the boundary conditions $y(\pi)=y(0)=0$, where $q(x)$ is a complex-valued
summable function. In~\cite{13}, we proved that the eigenvalue $\lambda_{n}$
of the operator $L(q)$ lying in $O(1)$ neighborhood of $n^{2}$ satisfies the
equality
\begin{equation}
(\lambda_{n}-n^{2}-A_{m}(\lambda_{n}))(\Psi_{n},\sin nx)=R_{m+1}, \label{eq2}%
\end{equation}
where $n$\ is a large number, $\Psi_{n}$ is an eigenfunction corresponding to
the eigenvalue $\lambda_{n}$ and satisfying the inequality $|(\Psi_{n},\sin
nx)|>1/2$,
\[
A_{m}(\lambda_{n})=\sum_{k=0}^{m}a_{k}(\lambda_{n}),
\]%
\begin{equation}
a_{0}(\lambda_{n})=-C_{2n},\qquad a_{1}(\lambda_{n})=\sum_{\substack{n_{1}%
=-\infty\\n_{1}\neq0,-2n}}^{\infty}\frac{C_{n_{1}}(C_{n_{1}}-C_{n_{1}+2n}%
)}{\lambda_{n}-(n+n_{1})^{2}}, \label{eq3}%
\end{equation}%
\begin{equation}
a_{2}(\lambda_{n})=\sum_{\substack{n_{1},n_{2}=-\infty\\n_{1},n_{1}+n_{2}%
\neq0,-2n}}^{\infty}\frac{C_{n_{1}}C_{n_{2}}(C_{n_{1}+n_{2}}-C_{n_{1}%
+n_{2}+2n})}{[\lambda_{n}-(n+n_{1})^{2}][\lambda_{n}-(n+n_{1}+n_{2})^{2}]},
\label{eq4}%
\end{equation}%
\begin{equation}
a_{k}(\lambda_{n})=\sum_{n_{1},n_{2},...,n_{k}=-\infty}^{\infty}\frac
{C_{n_{1}}C_{n_{2}}...C_{n_{k}}(C_{n_{1}+n_{2}+..+n_{k}}-C_{n_{1}%
+n_{2}+...+n_{k}+2n})}{\prod_{j=1}^{k}[\lambda_{n}-(n+\sum_{s=1}^{j}n_{s}%
)^{2}]}, \label{eq5}%
\end{equation}%
\begin{equation}
R_{m+1}(\lambda_{n})=\sum_{n_{1},n_{2},...,n_{m+1}=-\infty}^{\infty}%
\frac{C_{n_{1}}C_{n_{2}}...C_{n_{m+1}}\bigl(q(x)\Psi_{n}(x),\sin
\bigl(n+\sum_{k=1}^{m+1}n_{k}\bigr)\bigr)}{\prod_{j=1}^{m+1}[\lambda
_{n}-(n+\sum_{k=1}^{j}n_{k})^{2}]}, \label{eq6}%
\end{equation}
$C_{n}=\dfrac{1}{\pi}%
{\displaystyle\int_{0}^{\pi}}
q(x)\cos nxdx$, and, without loss of generality, it is assumed that $C_{0}=0$.
Here the sums for $a_{k}(\lambda_{n})$ and $R_{m+1}(\lambda_{n})$ are taken
under the conditions $n_{1}+n_{2}+...+n_{s}\neq0,-2n$, for $s=1,2,...,k$\ and
$s=1,2,...,m+1$, respectively. Moreover, we have
\begin{equation}
a_{k}=O\biggl(\bigl(\frac{ln\left\vert n\right\vert }{n}\bigr)^{k}%
\biggr),\qquad R_{m+1}=O\biggl(\bigl(\frac{ln\left\vert n\right\vert }%
{n}\bigr)^{m+1}\biggr). \label{eq7}%
\end{equation}
Using~\eqref{eq2}-\eqref{eq7}, we obtained the following asymptotic formulas:
\[
\lambda_{n}=n^{2}-C_{2n}+O\bigl(\frac{ln\left\vert n\right\vert }{n}\bigr),
\]

\begin{equation}
\lambda_{n}=n^{2}-C_{2n}+\sum_{\substack{k=-\infty\\k\neq0,-2n}}^{\infty}%
\frac{C_{k}(C_{k}-C_{k+2n})}{[n^{2}-(n+k)^{2}]}+O\biggl(\bigl(\frac
{ln\left\vert n\right\vert }{n}\bigr)^{2}\biggr), \label{eq8}%
\end{equation}
and other formulas of order $O((\ln n)^{l}n^{-l})$ (for all $l>0$) for
$\lambda_{n}$ and corresponding eigenfunctions of the operator $L(q)$.

In this paper, in Section~\ref{Sec:2}, we improve the asymptotic formulas,
instead of the series in~\eqref{eq8}, we write some Fourier coefficients of
some functions explicitly expressed by the potential $q$. For the potential of
bounded variation, we explicitly determine the first and second terms of the
asymptotic formulas, which have the orders $O(n^{-1})$\ and $O(n^{-2})$,
respectively (see~\eqref{eq30} and Theorem~\ref{t1}). Note that, generally
speaking, if the potential $q$\ is a piecewise continuous function with some
jump, then the first term $C_{2n}$\ is of order $n^{-1}$\ and the next term in
asymptotic formula~\eqref{eq30} is of order $n^{-2}$. Consequently, asymptotic
formula~\eqref{eq30} with error $O(n^{-3})$\ explicitly separates the first
and second terms from the error terms.\ Then, we demonstrate this formula for
the Kronig-Penney potential
\begin{equation}
q(x)=%
\begin{cases}
a & \text{if }x\in\lbrack0,c]\\[2mm]%
b & \text{if }x\in(c,\pi],
\end{cases}
\label{eq9}%
\end{equation}
where $q(x+\pi)=q(x)$, and $c\in(0,\pi)$. Without loss of generality, we
assume that $a<b$. This inequality with the condition $C_{0}=0$\ implies that
\begin{equation}
ac+b(\pi-c)=0,\qquad a<0<b,\text{ }(b-a)c=b\pi. \label{eq10}%
\end{equation}

Note that, the Kronig-Penney model is a simplified model for an electron in a
one-dimensional periodic potential and has been studied in many works (see,
for example,~\cite{1, 3, 8, 10, 11}, and references therein). Veliev~\cite{10,
11} studied the bands and gaps in the spectrum of the Schr\"{o}dinger operator
with the Kronig-Penney potential and obtained asymptotic formulas for the
length of the gaps in the spectrum.

In Section~\ref{Sec:3}, we provide estimates for small Dirichlet eigenvalues
of the Schr\"{o}dinger operator $L(q)$ with the Kronig-Penney potential.
First, we derive equation~\eqref{eq42} for calculating the Dirichlet
eigenvalues using some iterations formulas by the methods used in~\cite{7, 9,
12, 13}. It is important to note that, the main difficulty of this section is
to prove that all the eigenvalues (small and large) of $L(q)$ can be obtained
from~\eqref{eq42} by the fixed point iteration. Then, we examine the small
eigenvalues using numerical methods. Moreover, in Section~\ref{Sec:3}, as a
result of meticulous and detailed investigations, we establish conditions on
the potential under which all the small eigenvalues satisfy
equality~\eqref{eq42} and the Banach fixed point theorem can be used for this equality.

\section{On the Sharp Asymptotic Formulas}

\label{Sec:2} In this section, first, to obtain explicit formulas for the
large eigenvalues from~\eqref{eq8}, we consider the series in~\eqref{eq8} as
the difference of the series
\begin{equation}
A_{1,n}(n^{2}):=\sum_{\substack{k=-\infty\\k\neq0,-2n}}^{\infty}\frac
{C_{k}^{2}}{n^{2}-(n+k)^{2}},\qquad B_{1,n}(n^{2}):=\sum_{\substack{k=-\infty
\\k\neq0,-2n}}^{\infty}\frac{C_{k}C_{k+2n}}{n^{2}-(n+k)^{2}}. \label{eq11}%
\end{equation}
Then, we consider these series in the following lemmas. First, let us consider
$A_{1,n}(n^{2})$. For this, introduce the notations:%
\begin{equation}
Q(x,n)=\int_{0}^{x}q(t)e^{i2nt}\,dt-q_{-2n}x,\qquad Q_{n,0}=\frac{1}{\pi}%
\int_{0}^{\pi}Q(x,n)dx \label{eq12}%
\end{equation}
and%
\begin{equation}
G(x,n)=\int_{0}^{x}\bigl(q(t)e^{i2nt}\,+\frac{i\pi}{2}e^{it}\,q_{-2n}\bigr)dt,
\label{eq13}%
\end{equation}
where $q_{k}=\dfrac{1}{\pi}%
{\displaystyle\int_{0}^{\pi}}
q(x)e^{-ikx}dx$\ and $q_{0}=0$.

\begin{lemma}
\label{l1} If $q\in L_{1}[0,\pi]$, then the following formula holds:
\begin{equation}
A_{1,n}(n^{2})=\frac{C_{2n}^{2}}{4n^{2}}+D(n,q)+2\operatorname{Re}D(n,p),
\label{eq14}%
\end{equation}
where $D(n,f)=D_{1}(n,f)+D_{2}(n,f)$,
\[
D_{1}(n,f)=\frac{i}{4n\pi}\int_{0}^{\pi}f(x)(Q(x,n)-Q_{n,0})e^{-i2nx}%
\,dx,\quad D_{2}(n,f)=\frac{i}{4n\pi}\int_{0}^{\pi}f(x)G(x,n)e^{-i2nx}\,dx,
\]
and $p(x)=q(-x)$.
\end{lemma}

\begin{proof}
By the definition of $A_{1,n}(n^{2})$, we have
\[
A_{1,n}(n^{2})=-\sum_{\substack{k=-\infty\\k\neq0,-2n}}^{\infty}\frac
{C_{k}^{2}}{k(2n+k)}.
\]
Since $\dfrac{1}{k(2n+k)}=\dfrac{1}{2n}\bigl(\dfrac{1}{k}-\dfrac{1}%
{2n+k}\bigr)$, we see that
\begin{equation}
A_{1,n}(n^{2})=\frac{1}{2n}\sum_{k\neq0,-2n}\frac{C_{k}^{2}}{2n+k}-\frac
{1}{2n}\sum_{k\neq0,-2n}\frac{C_{k}^{2}}{k}. \label{eq15}%
\end{equation}
Using the obvious equality $C_{-k}=C_{k}$ and $C_{0}=0,$ we obtain the
following equality for the second term on the right side of~\eqref{eq15}:
\begin{equation}
\frac{1}{2n}\sum_{k\neq0,-2n}\frac{C_{k}^{2}}{k}=\frac{C_{2n}^{2}}{4n^{2}}.
\label{eq16}%
\end{equation}

Now let us estimate the first term on the right side of~\eqref{eq15}. Since
$q_{-k}=\overline{q_{k}}$ and $\cos kx=\dfrac{1}{2}(e^{ikx}+e^{-ikx})$, we
have $C_{k}^{2}=\dfrac{1}{4}(2\left\vert q_{k}\right\vert ^{2}+q_{k}%
^{2}+q_{-k}^{2})$. Therefore, the first term on the right side of~\eqref{eq15}
can be written in the form
\begin{equation}
\frac{1}{2n}\sum_{k\neq0,-2n}\frac{C_{k}^{2}}{2n+k}=I_{1}+I_{2}+I_{3},
\label{eq17}%
\end{equation}
where
\begin{equation}
I_{1}=\frac{1}{4n}\sum_{\substack{k=-\infty\\k\neq0,-2n}}^{\infty}%
\frac{\left\vert q_{k}\right\vert ^{2}}{(2n+k)}=I_{1,1}+I_{1,2},\qquad
I_{2}=\frac{1}{8n}\sum_{\substack{k=-\infty\\k\neq0,-2n}}^{\infty}\frac
{q_{-k}^{2}}{(2n+k)},\qquad I_{3}=\overline{I_{2}}, \label{eq18}%
\end{equation}%
\[
I_{1,1}=\frac{1}{4n}\sum_{\substack{m=-\infty\\m\neq0,-n}}^{\infty}%
\frac{\left\vert q_{2m}\right\vert ^{2}}{(2n+2m)},\qquad I_{1,2}=\frac{1}%
{4n}\sum_{\substack{m=-\infty\\}}^{\infty}\frac{\left\vert q_{2m+1}\right\vert
^{2}}{(2n+2m+1)}.
\]
First, let us prove that
\begin{equation}
I_{1,1}=D_{1}(n,q),\qquad I_{1,2}=D_{2}(n,q),\qquad I_{1}=D(n,q). \label{eq19}%
\end{equation}
It is clear that $Q(\pi,n)=0$, $Q(0,n)=0$, and that the derivative of $Q(x,n)$
with respect to $x$\ is $Q^{\prime}(x,n)=q(x)e^{i2nx}\,-q_{-2n}$. Using the
integration by parts, we see that
\[
\int_{0}^{\pi}Q(x,n)e^{-i(2n+2m)x}dx=\frac{\pi q_{2m}}{i(2n+2m)},
\]
for $n+m\neq0$. Therefore, the Fourier decomposition of $Q(x,n)$ by the
orthonormal basis $\{e^{i2mx}/\sqrt{\pi}:m\in\mathbb{Z\}}$ of $L_{2}[0,\pi]$
has the form
\[
Q(x,n)=Q_{n,0}+\sum_{m\neq0,-n}\frac{q_{2m}e^{i(2n+2m)x}}{i(2n+2m)}.
\]
Using this decomposition in the integral for $D_{1}(n,q)$, we obtain the proof
of the first equality of~\eqref{eq19}.

Now, let us prove the second equality of~\eqref{eq19}. It is clear that
\[
G(\pi,n)=\int_{0}^{\pi}\bigl(q(t)e^{i2nt}\,+\frac{i\pi}{2}e^{it}%
\,q_{-2n}\bigr)dt=\pi q_{-2n}+\frac{\pi}{2}(e^{i\pi}-1)q_{-2n}=0,
\]
$G(0,n)=0$ and that $G^{\prime}(x,n)=q(x)e^{i2nx}\,+\dfrac{i\pi}{2}%
e^{ix}q_{-2n}$. Therefore,
\[
\int_{0}^{\pi}G(x,n)e^{-i(2n+2m+1)x}dx=\frac{\pi q_{2m+1}}{i(2n+2m+1)}.
\]
Thus, the Fourier decomposition of $G(x,n)$ by the orthonormal basis
$\{e^{i(2m+1)x}/\sqrt{\pi}:m\in\mathbb{Z}\}$ of $L_{2}[0,\pi]$ has the form \
\[
G(x,n)=\sum_{m\in\mathbb{Z}}\frac{q_{2m+1}e^{i(2n+2m+1)x}\,}{i(2n+2m+1)}.
\]
Using this decomposition of $G(x,n)$ in the integral for $D_{2}(n,q)$, we
obtain the proof of the second equality of~\eqref{eq19}. The third equality
of~\eqref{eq19} follows from~\eqref{eq18}.

Now, we consider $I_{2}$. Using the definition of $p(x)$ and the substitution
$t=-x$, we obtain
\[
p_{k}:=\frac{1}{\pi}\int_{0}^{\pi}p(x)e^{-ikx}dx\,=\frac{1}{\pi}\int_{0}^{\pi
}q(-x)e^{-ikx}dx\,=\frac{1}{\pi}\int_{0}^{\pi}q(t)e^{ikt}dt=q_{-k}.
\]
Therefore, instead of $q(x)$, using $p(x)$ and repeating the proof
of~\eqref{eq19}, we obtain that
\begin{equation}
I_{2}=D(n,p),\qquad I_{2}+I_{3}=2\operatorname{Re}D(n,p). \label{eq20}%
\end{equation}
Thus, the proof of~\eqref{eq14} follows from~\eqref{eq15}-\eqref{eq20}. The
lemma is proved.
\end{proof}

Now let us consider $B_{1,n}(n^{2})$.

\begin{lemma}
\label{l2} If $q\in L_{1}[0,\pi]$, then the following formula holds
\begin{equation}
B_{1,n}(n^{2})=-\frac{1}{\pi}\int_{0}^{\pi}Q^{2}(x)\cos2nxdx, \label{eq21}%
\end{equation}
where $Q(x)=$ $%
{\displaystyle\int_{0}^{x}}
q(x)dx$.
\end{lemma}

\begin{proof}
To prove~\eqref{eq21}, we estimate the left and right sides of this equality
separately. First, estimate the right-side. Let $Q_{k}=%
{\displaystyle\int_{0}^{\pi}}
Q(x)\sin kxdx$. Then, $Q_{k}=\pi C_{k}/k$. Therefore, we have the
decomposition
\[
Q(x)=\sum_{k=1}^{\infty}\dfrac{2}{\pi}Q_{k}\sin kx
\]
of $Q(x)$. Using the equality $(Q^{2}(x))^{\prime}=2q(x)Q(x)$, integration by
parts, the decomposition of $Q(x)$, and then the equality $\sin kx\sin
2nx=\dfrac{1}{2}(\cos(2n-k)x-\cos(2n+k)x)$, we obtain
\begin{align*}
&  \int_{0}^{\pi}Q^{2}(x)\cos2nxdx=-\frac{1}{n}\int_{0}^{\pi}q(x)Q(x)\sin
2nxdx\\
&  \qquad=-\frac{1}{\pi n}\int_{0}^{\pi}q(x)\sum_{k=1}^{\infty}Q_{k}%
(\cos(2n-k)x-\cos(2n+k)x)dx\\
&  \qquad=\frac{-1}{n}\sum_{k=1}^{\infty}Q_{k}(C_{2n-k}-C_{2n+k})=\frac{-1}%
{n}\sum_{k=1}^{\infty}Q_{k}C_{2n-k}+\frac{1}{n}\sum_{k=1}^{\infty}%
Q_{k}C_{2n+k}).
\end{align*}
Now, using the equalities $Q_{k}=\pi C_{k}/k$, $C_{-k}=C_{k}$, and doing the
substitution $k\rightarrow-k$\ in the second summation of the last equality,
we obtain%
\[
\int_{0}^{\pi}Q^{2}(x)\cos2nxdx=\frac{-\pi}{n}\sum_{k\in\mathbb{Z},k\neq
0}\frac{C_{k}C_{2n-k}}{k}.
\]
Thus, the right side of~\eqref{eq21} has the form
\[
-\frac{1}{\pi}\int_{0}^{\pi}Q^{2}(x)\cos2nxdx=\frac{1}{n}\sum_{k\in
\mathbb{Z},k\neq0,2n}\frac{C_{k}C_{2n-k}}{k},
\]
since $C_{0}=0$.

Now, we prove that the left side of~\eqref{eq21} also can be written in this
form. Doing the substitution $k\rightarrow-k$, and then using the equality
$\dfrac{1}{k(2n-k)}=\dfrac{1}{2n}\bigl(\dfrac{1}{k}+\dfrac{1}{2n-k}\bigr)$, we
see that
\begin{align*}
&  B_{1,n}(n^{2})=\sum_{k\in\mathbb{Z},k\neq0,-2n}\frac{C_{k}C_{2n+k}}%
{n^{2}-(n+k)^{2}}=\sum_{k\in\mathbb{Z},k\neq0,-2n}\frac{C_{k}C_{2n+k}%
}{(2n+k)(-k)}=\sum_{k\in\mathbb{Z},k\neq0,2n}\frac{C_{k}C_{2n-k}}{k(2n-k)}\\
&  \qquad=\frac{1}{2n}\sum_{k\in\mathbb{Z},k\neq0,2n}\frac{C_{k}C_{2n-k}}%
{k}+\frac{1}{2n}\sum_{k\in\mathbb{Z},k\neq0,2n}\frac{C_{k}C_{2n-k}}{2n-k}.
\end{align*}
On the other hand, doing the substitution $k\rightarrow2n-k$\ in the second
summation of the last equality, we obtain
\[
\sum_{k\in\mathbb{Z},k\neq0,2n}\frac{C_{k}C_{2n-k}}{2n-k}=\sum_{k\in
\mathbb{Z},k\neq0,2n}\frac{C_{k}C_{2n-k}}{k},\qquad B_{1,n}(n^{2})=\frac{1}%
{n}\sum_{k\in\mathbb{Z},k\neq0,2n}\frac{C_{k}C_{2n-k}}{k},
\]
which means that~\eqref{eq21} holds. The lemma is proved.
\end{proof}

Now, we consider $a_{2}(n^{2})$. It is clear that $a_{2}(n^{2})=a_{2,1}%
(n^{2})-a_{2,2}(n^{2})$, where
\begin{align*}
a_{2,1}(n^{2})  &  =\sum\limits_{k,k+l\neq0,-2n}\frac{C_{k}C_{l}C_{k+l}%
}{[n^{2}-(n+k)^{2}][n^{2}-((n+k+l))^{2}]}=\sum\limits_{k,k+l\neq0,-2n}%
\frac{C_{k}C_{l}C_{-k-l}}{k(2n+k)(k+l)(2n+k+l)},\\
a_{2,2}(n^{2})  &  =\sum\limits_{k,k+l\neq0,-2n}\frac{C_{k}C_{l}C_{2n+k+l}%
}{k(2n+k)(k+l)(2n+k+l)}.
\end{align*}
Using~the equality $\dfrac{1}{(k+l)(2n+k+l)}=\dfrac{1}{2n}\bigl(\dfrac{1}%
{k+l}+\dfrac{1}{2n+k+l}\bigr)$, for $l=0$ and $l\neq0$, we obtain
\begin{equation}
a_{2,1}(n^{2})=\frac{1}{4n^{2}}(S_{1}(n,q)+S_{2}(n,q)+S_{3}(n,q)+S_{4}(n,q)),
\label{eq22}%
\end{equation}
where
\begin{align*}
S_{1}(n,q)  &  =\sum\limits_{k,k+l\neq0,-2n}\frac{C_{k}C_{l}C_{-k-l}}%
{k(k+l)},\qquad S_{2}(n,q)=\sum\limits_{k,k+l\neq0,-2n}\frac{C_{k}%
C_{l}C_{-k-l}}{(2n+k)(k+l)},\\
S_{3}(n,q)  &  =\sum\limits_{k,k+l\neq0,-2n}\frac{C_{k}C_{l}C_{-k-l}%
}{k(2n+k+l)},\qquad S_{4}(n,q)=\sum\limits_{k,k+l\neq0,2n}\frac{C_{k}%
C_{l}C_{-k-l}}{(2n+k)(2n+k+l)}.
\end{align*}
First, let us consider $S_{1}(n,q)$. Grouping the terms
\begin{align*}
&  \frac{C_{k}C_{l}C_{-k-l}}{k(k+l)},\quad\frac{C_{l}C_{k}C_{-k-l}}%
{l(k+l)},\quad\frac{C_{k}C_{-k-l}C_{l}}{k(-l)},\\
&  \frac{C_{l}C_{-k-l}C_{k}}{l(-k)},\quad\frac{C_{-k-l}C_{l}C_{k}}%
{(-k-l)(-k)},\quad\frac{C_{-k-l}C_{k}C_{l}}{(-k-l)(-l)}%
\end{align*}
with the equal multiplicands, $C_{k}C_{l}C_{-k-l}$, $C_{l}C_{k}C_{-k-l}$,
$C_{k}C_{-k-l}C_{l}$, $C_{l}C_{-k-l}C_{k}$, $C_{-k-l}C_{l}C_{k}$,
$C_{-k-l}C_{k}C_{l}$ and using the equality $\dfrac{1}{k(k+l)}+\dfrac
{1}{l(k+l)}=\dfrac{1}{kl}$, we obtain that
\begin{equation}
\sum\limits_{k,l}\frac{C_{k}C_{l}C_{-k-l}}{k(k+l)}=0, \label{eq23}%
\end{equation}
where the summations are taken under conditions $k,l,k+l\neq0$, since
$C_{0}=0$. However, in summation for $S_{1}(n,q)$ it is assumed additionally
that $k,k+l\neq-2n$. Therefore, we also need to consider the following
expressions:
\[
S_{1,1}(n,q):=\frac{C_{2n}}{2n}\sum\limits_{l\neq0,-2n}\frac{C_{2n}%
C_{l}C_{l-2n}}{2n(l-2n)},\qquad S_{1,2}(n,q)=\frac{C_{2n}}{2n}\sum
\limits_{l\neq0,-2n}\frac{C_{2n+l}C_{l}}{(2n+l)}.
\]
Using the Schwarz inequality for $l_{2}$ and taking into account that
$C_{l}=O(1)$, we obtain
\begin{equation}
S_{1,1}(n,q):=\frac{C_{2n}}{2n}O(1)=o\bigl(\frac{1}{n}\bigr),\qquad
S_{1,2}(q):=\frac{C_{2n}}{2n}O(1)=o\bigl(\frac{1}{n}\bigr), \label{eq24}%
\end{equation}
for $q\in L_{2}[0,\pi]$.

Denote by $E$ the set of all functions satisfying
\begin{equation}
C_{n}=\frac{1}{\pi}\int_{0}^{\pi}q(x)\cos nxdx=O\bigl(\frac{1}{n}\bigr),
\label{eq25}%
\end{equation}
as $n\rightarrow\infty$. For example, if $q$ is a function of bounded
variation on $[0,\pi]$, then~\eqref{eq25} holds. There is a larger set of
functions satisfying~\eqref{eq25} (see~\cite{14}). Here we do not discuss the
set $E$, since this is not the aim of this paper. We will apply the obtained
further results for potential~\eqref{eq9} which is a bounded variation.

Using~\eqref{eq25} in~\eqref{eq24}, we obtain
\[
S_{1,1}(n,q)=O\bigl(\frac{1}{n^{2}}\bigr),\text{ }S_{1,2}(q)=O\bigl(\frac
{1}{n^{2}}\bigr).
\]
These equalities with~\eqref{eq23} imply the following:

\begin{proposition}
\label{p1}If $q\in E$, then $S_{1}(n,q)=O(n^{-2})$.
\end{proposition}

Now, we estimate $S_{j}(n,q)$, for $j=2,3,4$, in the following lemma. For
this, we use the following easily checkable relations:%
\begin{equation}
\sum\limits_{k\neq0,-2n}\frac{1}{k^{2}(2n+k)}=O\bigl(\frac{1}{n}%
\bigr),\quad\sum\limits_{l\neq0,-k}\frac{1}{l(k+l)^{2}}=O\bigl(\frac{1}%
{k}\bigr),\quad\sum\limits_{k\neq0,-2n}\frac{1}{\left\vert k(2n+k)\right\vert
}=O\bigl(\frac{\ln n}{n}\bigr). \label{eq26}%
\end{equation}
To prove the first equality, split the sum into three parts
\[
\sum\limits_{k\leq n,\text{ }k\neq0}\frac{1}{k^{2}(2n+k)},\qquad
\sum\limits_{k>n}\frac{1}{k^{2}(2n+k)},\qquad\sum\limits_{k<-n,k\neq-2n}%
\frac{1}{k^{2}(2n+k)}%
\]
and estimate each of them in a standard way. The second and third equalities
can be proved in the same way.

\begin{lemma}
\label{l3} If $q\in E$, then the equalities
\begin{equation}
S_{j}(n,q)=O\bigl(\frac{1}{n}\bigr), \label{eq27}%
\end{equation}
for $j=2,3,4$, hold.
\end{lemma}

\begin{proof}
It follows from~\eqref{eq25} and the definition of $S_{2}(n,q)$ that
\[
\left\vert -S_{2}(n,q)\right\vert =O(1)\biggl(\sum\limits_{k\neq0,-2n}%
\sum\limits_{l\neq0,-k}\frac{1}{\left\vert k(2n+k)l(k+l)^{2}\right\vert
}\biggr).
\]
Now using~\eqref{eq26}, we obtain
\[
\left\vert -S_{2}(n,q)\right\vert =O(1)\sum\limits_{k\neq0,-2n}\frac
{1}{\left\vert k^{2}(2n+k)\right\vert }=O\bigl(\frac{1}{n}\bigr).
\]
Now, let us prove~\eqref{eq27} for $j=3$. By~\eqref{eq25}, we have
\[
\left\vert -S_{3}(n,q)\right\vert =O(1)\sum\limits_{k\neq0,-2n}\sum
\limits_{l\neq0,-k}\frac{1}{\left\vert l(2n+k+l)k^{2}(k+l)\right\vert }.
\]
Making the substitution $s=k+l$, we obtain
\[
\left\vert -S_{3}(n,q)\right\vert =O(1)\sum\limits_{s\neq0,-2n}\sum
\limits_{k\neq0,s}\frac{1}{\left\vert s(2n+s)k^{2}(s-k)\right\vert }.
\]
Now using~\eqref{eq26}, we obtain
\[
\left\vert -S_{3}(n,q)\right\vert =O(1)\sum\limits_{s\neq0,2n}\frac
{1}{\left\vert s^{2}(2n+s)\right\vert }=O\bigl(\frac{1}{n}\bigr),
\]
that is, the proof of~\eqref{eq27} for $j=3$. Arguing as above, we conclude
that
\[
-S_{4}(n,q)=O(1)\biggl(\sum\limits_{k\neq0,-2n}\frac{1}{\left\vert
k(2n+k)\right\vert }\biggr)\biggl(\sum\limits_{s\neq0,-2n}\frac{1}{\left\vert
s(2n+s)\right\vert }\biggr).
\]
Therefore,~\eqref{eq26} implies that
\[
-S_{4}(n,q)=O\bigl(\frac{\left(  \ln n\right)  ^{2}}{n^{2}}\bigr)=O\bigl(\frac
{1}{n}\bigr).
\]
The lemma is proved.
\end{proof}

It follows from Proposition~\ref{p1} and Lemma~\ref{l3} that%
\begin{equation}
a_{2,1}(n^{2})=O\bigl(\frac{1}{n^{3}}\bigr), \label{eq28}%
\end{equation}
if $q\in E$. Now, let us estimate $a_{2,2}(n^{2})$. Arguing as in the proof
of~\eqref{eq22}, we see that $a_{2,2}(n^{2})=\dfrac{1}{4n^{2}}(H_{1}%
+H_{2}+H_{3}+H_{4})$, where%
\begin{align*}
H_{1}(n,q)  &  =\sum\limits_{k,k+l\neq0,-2n}\frac{C_{k}C_{l}C_{2n+k+l}%
}{k(k+l)},\qquad H_{2}(n,q)=\sum\limits_{k,k+l\neq0,-2n}\frac{C_{k}%
C_{l}C_{2n+k+l}}{(2n+k)(k+l)},\\
H_{3}(n,q)  &  =\sum\limits_{k,k+l\neq0,-2n}\frac{C_{k}C_{l}C_{2n+k+l}%
}{k(2n+k+l)},\qquad H_{4}(n,q)=\sum\limits_{k,k+l\neq0,2n}\frac{C_{k}%
C_{l}C_{2n+k+l}}{(2n+k)(2n+k+l)}.
\end{align*}
Therefore, arguing as in the proof of Lemma~\ref{l3}, we conclude that
\begin{equation}
a_{2,2}(n^{2},q)=O\bigl(\frac{1}{n^{3}}\bigr), \label{eq29}%
\end{equation}
if $q\in E$.

Now, using Lemmas~\ref{l1} and~\ref{l2}, \eqref{eq28} and~\eqref{eq29}, we
obtain the following main results of this section:

\begin{theorem}
\label{t1}If $q\in E$, then the following asymptotic formula holds:
\begin{equation}
\lambda_{n}=n^{2}-C_{2n}+\frac{1}{\pi}\int_{0}^{\pi}Q^{2}(x)\cos
2nxdx+D(n,q)+2\operatorname{Re}D(n,p)+O\bigl(\frac{1}{n^{3}}\bigr),
\label{eq30}%
\end{equation}
where $D(n,f)$ and $Q(x)$ are defined in Lemma~\ref{l1} and Lemma~\ref{l2}.
\end{theorem}

\begin{proof}
It follows from~\eqref{eq8},~\eqref{eq25} and~\eqref{eq26} that $\lambda
_{n}=n^{2}+O\bigl(\dfrac{1}{n}\bigr)$. Using this, we easily get the estimate
\[
\sum_{\substack{k\neq0,-2n}}^{\infty}\biggl|\frac{1}{\lambda_{n}-(n+k)^{2}%
}-\frac{1}{n^{2}-(n+k)^{2}}\biggr|=O\bigl(\frac{1}{n^{3}}\bigr).
\]
By virtue of this estimate
\[
a_{j}\left(  \lambda_{n}\right)  =a_{j}\left(  n^{2}\right)  +O\bigl(\frac
{1}{n^{3}}\bigr),
\]
for $j=1,2,\ldots$. Therefore, it follows from~\eqref{eq2}-\eqref{eq6} that
\begin{equation}
\lambda_{n}=n^{2}-C_{2n}+a_{1}\left(  n^{2}\right)  +a_{2}\left(
n^{2}\right)  +a_{3}\left(  n^{2}\right)  +O\bigl(\frac{1}{n^{3}}\bigr),
\label{eq31}%
\end{equation}
where $a_{1}(n^{2})=A_{1,n}(n^{2})-B_{1,n}(n^{2})$, equivalently
\begin{equation}
a_{1}\left(  n^{2}\right)  =\frac{1}{\pi}\int_{0}^{\pi}Q^{2}(x)\cos
2nxdx+D(n,q)+2\operatorname{Re}D(n,p)+O\bigl(\frac{1}{n^{3}}%
\bigr) \label{eq32}%
\end{equation}
(see Lemma~\ref{l1}, Lemma~\ref{l2} and~\eqref{eq25}), and
\begin{equation}
a_{2}\left(  n^{2}\right)  =O\bigl(\frac{1}{n^{3}}\bigr) \label{eq33}%
\end{equation}
(see~\eqref{eq28} and~\eqref{eq29}). On the other hand, we have
\[
a_{3}\left(  n^{2}\right)  =\sum_{n_{1},n_{2},...,n_{k}=-\infty}^{\infty}%
\frac{C_{k}C_{l}C_{s}(C_{k+l+s}-C_{2n+k+l+s})}%
{k(2n+k)(k+l)(2n+k+l)(k+l+s)(2n+k+l+s)}.
\]
Moreover, arguing as in the proof of~\eqref{eq33}, we see that $a_{3}\left(
n^{2}\right)  =O\bigl(\dfrac{1}{n^{3}}\bigr)$. Thus, the proof of the theorem
follows from~\eqref{eq31}-\eqref{eq33}.
\end{proof}

Now, let us demonstrate the obtained results for potential~\eqref{eq9}. In
this case, by direct calculation, we have
\begin{equation}
q_{k}=\dfrac{1}{\pi}\int_{0}^{\pi}q(x)e^{-ikx}dx\,=\dfrac{1}{i\pi
k}((b-a)e^{-ikc}+a-(-1)^{k}b). \label{eq34}%
\end{equation}
Using the equalities $Q(\pi,n)=G(\pi,n)=0$ (see the proof of Lemma~\ref{l1}),
we obtain
\[
Q(x,n)=%
\begin{cases}
\dfrac{a}{i2n}(e^{i2nx}-1)-q_{-2n}x & \text{if }x\in\lbrack0,c]\\[2mm]%
\dfrac{b}{i2n}(e^{i2nx}-1)-q_{-2n}x+\pi q_{-2n} & \text{if }x\in(c,\pi]
\end{cases}
\]
and
\[
G(x,n)=%
\begin{cases}
\dfrac{a}{i2n}(e^{i2nx}-1)+\dfrac{\pi}{2}q_{-2n}x(e^{ix}-1) & \text{if }%
x\in\lbrack0,c]\\[2mm]%
\dfrac{b}{i2n}(e^{i2nx}-1)+\dfrac{\pi}{2}q_{-2n}x(e^{ix}-1)+\pi q_{-2n} &
\text{if }x\in(c,\pi],
\end{cases}
\]
where $Q(x,n)$, and $G(x,n)$\ are defined by~\eqref{eq12} and~\eqref{eq13}.
Therefore, by Lemma~\ref{l1}, we have
\begin{align*}
&  D_{1}(n,q)=\frac{i}{4\pi n}\int_{0}^{\pi}q(x)Q(x,n)e^{-i2nx}\,dx-\frac
{iQ_{n,0}}{4\pi n}\int_{0}^{\pi}q(x)e^{-i2nx}\,dx\\
&  \qquad=\frac{i}{4\pi n}\biggl(\int_{0}^{c}q(x)Q(x,n)e^{-i2nx}\,dx+\int
_{c}^{\pi}q(x)Q(x,n)e^{-i2nx}\,dx\biggr)-\frac{iQ_{n,0}q_{2n}}{4n},
\end{align*}
and
\begin{align*}
&  D_{2}(n,q)=\frac{i}{4\pi n}\int_{0}^{\pi}q(x)G(x,n)e^{-i2nx}\,dx\\
&  \qquad=\frac{i}{4\pi n}\biggl(\int_{0}^{c}q(x)G(x,n)e^{-i2nx}\,dx+\int
_{c}^{\pi}q(x)G(x,n)e^{-i2nx}\,dx\biggr).
\end{align*}
We first calculate the terms in $D_{1}(n,q)$. Using $(b-a)c=b\pi$ or
$ac=b(c-\pi)$\ by~\eqref{eq10} and substituting $q_{-2n}=\dfrac
{i(b-a)(e^{i2nc}-1)}{2\pi n}$ by~\eqref{eq34},\ we have
\begin{align}
\pi Q_{n,0}  &  =\int_{0}^{c}\bigl(\dfrac{a}{i2n}(e^{i2nx}-1)-q_{-2n}%
x\bigr)dx+\int_{c}^{\pi}\bigl(\dfrac{b}{i2n}(e^{i2nx}-1)-q_{-2n}x+\pi
q_{-2n}\bigr)dx\nonumber\\
&  =\frac{(b-a)(e^{i2nc}-1)}{4n^{2}}+\pi(\frac{\pi}{2}-c)q_{-2n}%
=\frac{(b-a)(e^{i2nc}-1)}{4n^{2}}-\frac{i\pi(b+a)(e^{i2nc}-1)}{4n},\nonumber
\end{align}%
\[
-\frac{i}{4n}Q_{n,0}q_{2n}=\frac{i(b^{2}-a^{2})(1-\cos2nc)}{16\pi n^{3}}%
-\frac{(b-a)^{2}(1-\cos2nc)}{16\pi^{2}n^{4}},
\]%
\begin{align*}
&  \int_{0}^{c}q(x)Q(x,n)e^{-i2nx}\,dx=\int_{0}^{c}a\biggl(\dfrac{a}%
{i2n}(e^{i2nx}-1)-q_{-2n}x\biggr)e^{-i2nx}\,dx\\
&  \qquad=\dfrac{a^{2}}{i2n}\biggl(c+\dfrac{e^{-i2nc}-1}{i2n}\biggr)+\dfrac
{aq_{-2n}}{i2n}\biggl(ce^{-i2nc}+\dfrac{e^{-i2nc}-1}{i2n}\biggr),
\end{align*}
and
\begin{align*}
&  \int_{c}^{\pi}q(x)Q(x,n)e^{-i2nx}\,dx=\int_{c}^{\pi}b\biggl(\dfrac{b}%
{i2n}(e^{i2nx}-1)-q_{-2n}x+\pi q_{-2n}\biggr)e^{-i2nx}\,dx\\
&  \qquad=\dfrac{b^{2}}{i2n}\biggl(\pi-c-\dfrac{e^{-i2nc}-1}{i2n}%
\biggr)+\dfrac{bq_{-2n}}{i2n}\biggl(\pi-ce^{-i2nc}-\dfrac{e^{-i2nc}-1}%
{i2n}\biggr)+\dfrac{b\pi q_{-2n}}{i2n}(e^{-i2nc}-1).
\end{align*}
Therefore,
\begin{align*}
&  D_{1}(n,q)=\frac{i}{4\pi n}\int_{0}^{\pi}q(x)Q(x,n)e^{-i2nx}\,dx-\frac
{iQ_{n,0}}{4\pi n}\int_{0}^{\pi}q(x)e^{-i2nx}\,dx\\
&  \qquad=\frac{i}{4\pi n}\biggl(\frac{(b^{2}-a^{2})(e^{-i2nc}-1)}{4n^{2}%
}-\frac{ab\pi}{i2n}+\frac{(b-a)(e^{-i2nc}-1)q_{-2n}}{4n^{2}}\biggr)-\frac
{iQ_{n,0}q_{2n}}{4n}%
\end{align*}
and by~\eqref{eq34}, we have
\begin{equation}
D_{1}(n,q)=-\frac{ab}{8n^{2}}+\frac{(b^{2}-a^{2})\sin2nc}{16\pi n^{3}}%
-\frac{(b-a)^{2}(1-\cos2nc)}{8\pi^{2}n^{4}}. \label{eq35}%
\end{equation}
Now let us calculate the terms in $D_{2}(n,q)$:%

\begin{align*}
&  \int_{0}^{c}q(x)G(x,n)e^{-i2nx}\,dx=\int_{0}^{c}a\biggl(\dfrac{a}%
{i2n}(e^{i2nx}-1)+\dfrac{\pi}{2}q_{-2n}x(e^{ix}-1)\biggr)e^{-i2nx}\,dx\\
&  \qquad=\dfrac{a^{2}}{i2n}\int_{0}^{c}(1-e^{-i2nx})\,dx+a\dfrac{\pi}%
{2}q_{-2n}\int_{0}^{c}xe^{-i(2n-1)x}\,dx-a\dfrac{\pi}{2}q_{-2n}\int_{0}%
^{c}xe^{-i2nx}\,dx\\
&  \qquad=\dfrac{a^{2}}{i2n}\biggl(c+\dfrac{e^{-i2nc}-1}{i2n}\biggr)-\dfrac
{a\pi q_{-2n}}{i2(2n-1)}\biggl(ce^{-i(2n-1)c}+\dfrac{e^{-i(2n-1)c}-1}%
{i(2n-1)}\biggr)\\
&  \qquad+\dfrac{a\pi q_{-2n}}{i4n}\biggl(ce^{-i2nc}+\dfrac{e^{-i2nc}-1}%
{i2n}\biggr)
\end{align*}
and
\begin{align*}
&  \int_{c}^{\pi}q(x)G(x,n)e^{-i2nx}\,dx=\int_{c}^{\pi}b\biggl(\dfrac{b}%
{i2n}(e^{i2nx}-1)+\dfrac{\pi}{2}q_{-2n}x(e^{ix}-1)+\pi q_{-2n}\biggr)e^{-i2nx}%
\,dx\\
&  \qquad=\dfrac{b^{2}}{i2n}\int_{c}^{\pi}(1-e^{-i2nx})\,dx+b\dfrac{\pi}%
{2}q_{-2n}\int_{c}^{\pi}xe^{-i(2n-1)x}\,dx-b\dfrac{\pi}{2}q_{-2n}\int_{c}%
^{\pi}xe^{-i2nx}\,dx+b\pi q_{-2n}\int_{c}^{\pi}e^{-i2nx}\,dx\\
&  \qquad=\dfrac{b^{2}}{i2n}\biggl(\pi-c-\dfrac{e^{-i2nc}-1}{i2n}%
\biggr)+\dfrac{b\pi q_{-2n}}{i2(2n-1)}\biggl(\pi+ce^{-i(2n-1)c}+\dfrac
{e^{-i(2n-1)c}+1}{i(2n-1)}\biggr)\\
&  \qquad+\dfrac{b\pi q_{-2n}}{i4n}\biggl(\pi-ce^{-i2nc}-\dfrac{e^{-i2nc}%
-1}{i2n}\biggr)+\dfrac{b\pi q_{-2n}}{i2n}(e^{-i2nc}-1)
\end{align*}
give
\begin{align*}
&  D_{2}(n,q)=\frac{i}{4\pi n}\int_{0}^{\pi}q(x)G(x,n)e^{-i2nx}\,dx\\
&  \qquad=\frac{i}{4\pi n}\biggl(\frac{(b^{2}-a^{2})(e^{-i2nc}-1)}{4n^{2}%
}-\frac{ab\pi}{i2n}+\frac{b\pi^{2}(1+e^{-i(2n-1)c})q_{-2n}}{i2(2n-1)}%
+\frac{b\pi^{2}(1-e^{-i2nc})q_{-2n}}{i4n}\\
&  \qquad-\frac{\pi(b-a)(1-e^{-i2nc})q_{-2n}}{8n^{2}}-\frac{\pi q_{-2n}%
(b+a+(b-a)e^{-i(2n-1)c})}{2(2n-1)^{2}}\biggr).
\end{align*}
Using~\eqref{eq34}, we have
\begin{align}
&  D_{2}(n,q)=-\frac{ab}{8n^{2}}+\frac{i(b^{2}-a^{2})(e^{-i2nc}-1)}{16\pi
n^{3}}-\frac{ib(b-a)(1-e^{-i2nc})(e^{i2nc}+e^{ic})}{16n^{2}(2n-1)}%
+\frac{ib(b-a)(\cos2nc-1)}{16n^{3}}\nonumber\\
&  \qquad+\frac{(b-a)^{2}(\cos2nc-1)}{32\pi n^{4}}+\frac{(b-a)(e^{-i2nc}%
-1)[(b+a)e^{i2nc}+(b-a)e^{ic}]}{16\pi n^{2}(2n-1)^{2}}. \label{eq36}%
\end{align}
Therefore, using~\eqref{eq35},~\eqref{eq36}, Lemma~\ref{l1}, and simplifying a
bit, we obtain
\begin{align}
&  D(n,q)=D_{1}(n,q)+D_{2}(n,q)=-\frac{ab}{4n^{2}}+\frac{(b^{2}-a^{2})\sin
2nc}{8\pi n^{3}}-\frac{(b-a)^{2}(1-\cos2nc)}{8\pi^{2}n^{4}}\nonumber\\
&  \qquad+\frac{i(b^{2}-a^{2})(\cos2nc-1)}{16\pi n^{3}}+\frac{ib(b-a)(\cos
2nc-1)}{16n^{3}}-\frac{ib(b-a)(1-e^{-i2nc})(e^{i2nc}+e^{ic})}{16n^{2}%
(2n-1)}\nonumber\\
&  \qquad+\frac{(b-a)^{2}(\cos2nc-1)}{32\pi n^{4}}+\frac{(b-a)(e^{-i2nc}%
-1)[(b+a)e^{i2nc}+(b-a)e^{ic}]}{16\pi n^{2}(2n-1)^{2}}. \label{eq37}%
\end{align}

Now, let us calculate $\operatorname{Re}D(n,p)$, where $p(x)$ is defined in
Lemma~\ref{l1}. For the Kronig-Penney potential, $p(x)$ becomes $p(x)=b$ for
$x\in\lbrack0,\pi-c)$ and $p(x)=a$ for $x\in\lbrack\pi-c,\pi]$. Therefore, we
have
\begin{align*}
&  D_{1}(n,p)=\frac{i}{4\pi n}\int_{0}^{\pi}p(x)Q(x,n)e^{-i2nx}\,dx-\frac
{iQ_{n,0}}{4\pi n}\int_{0}^{\pi}p(x)e^{-i2nx}\,dx\\
&  \qquad=\frac{i}{4\pi n}\biggl(\int_{0}^{\pi-c}p(x)Q(x,n)e^{-i2nx}%
\,dx+\int_{\pi-c}^{\pi}p(x)Q(x,n)e^{-i2nx}\,dx\biggr)-\frac{iQ_{n,0}p_{2n}%
}{4n}%
\end{align*}
and
\[
D_{2}(n,p)=\frac{i}{4\pi n}\biggl(\int_{0}^{\pi-c}p(x)G(x,n)e^{-i2nx}%
\,dx+\int_{\pi-c}^{\pi}p(x)G(x,n)e^{-i2nx}\,dx\biggr).
\]
By direct calculations, similar to the calculations above for the terms in
$D_{1}(n,q)$ and $D_{2}(n,q)$, we obtain
\[
-\frac{i}{4n}Q_{n,0}p_{2n}=-\frac{i(b^{2}-a^{2})(e^{i2nc}-1)^{2}}{32\pi n^{3}%
}+\frac{(b-a)^{2}(e^{i2nc}-1)^{2}}{32\pi^{2}n^{4}},
\]%
\begin{equation}
D_{1}(n,p)=\frac{ab}{8n^{2}}-\frac{i(b-a)(b+3a)(e^{i2nc}-1)^{2}}{32\pi n^{3}%
}+\frac{(b-a)^{2}(e^{i2nc}-1)^{2}}{16\pi^{2}n^{4}}, \label{eq38}%
\end{equation}
and
\begin{align}
&  D_{2}(n,p)=\frac{ab}{8n^{2}}+\frac{i(b-a)(1-e^{i(2n-1)c})(e^{i2nc}%
-1)}{16n^{2}(2n-1)}+\frac{(b^{2}-a^{2}+(b-a)^{2}e^{i(2n-1)c})(e^{i2nc}%
-1)}{16\pi n^{2}(2n-1)^{2}}\nonumber\\
&  \qquad-\frac{ia(b-a)(e^{i2nc}-1)^{2}}{32n^{3}}+\frac{(b-a)^{2}%
(e^{i2nc}-1)^{2}}{64\pi n^{4}}+\frac{ia(b-a)(e^{i2nc}-1)^{2}}{16\pi n^{3}}.
\label{eq39}%
\end{align}
Thus, using~\eqref{eq38},~\eqref{eq39}, Lemma~\ref{l1}, and simplifying a bit,
we write
\begin{align}
&  D(n,p)=D_{1}(n,p)+D_{2}(n,p)=\frac{ab}{4n^{2}}+\frac{i(a^{2}-b^{2}%
)(e^{i2nc}-1)^{2}}{32\pi n^{3}}\nonumber\\
&  \qquad+\frac{i(b-a)(1-e^{i(2n-1)c})(e^{i2nc}-1)}{16n^{2}(2n-1)}%
-\frac{ia(b-a)(e^{i2nc}-1)^{2}}{32n^{3}}\nonumber\\
&  \qquad+\frac{(b-a)^{2}(e^{i2nc}-1)^{2}}{64\pi n^{4}}+\frac{(b-a)^{2}%
(e^{i2nc}-1)^{2}}{16\pi^{2}n^{4}}+\frac{(b^{2}-a^{2}+(b-a)^{2}e^{i(2n-1)c}%
)(e^{i2nc}-1)}{16\pi n^{2}(2n-1)^{2}}. \label{eq40}%
\end{align}
Now, we calculate $B_{1,n}(n^{2})$\ defined by~\eqref{eq21}. Expressing $Q(x)$
defined in Lemma~\ref{l2} for the Kronig-Penney potential as%
\[
Q(x)=%
\begin{cases}
ax & \text{if }x\in\lbrack0,c]\\[2mm]%
bx-b\pi & \text{if }x\in(c,\pi],
\end{cases}
\]
by direct calculation, we obtain
\begin{equation}
B_{1,n}(n^{2})=-\frac{1}{\pi}\int_{0}^{\pi}Q^{2}(x)\cos2nxdx=\frac{ab\cos
2nc}{2n^{2}}-\frac{(b^{2}-a^{2})\sin2nc}{4\pi n^{3}}. \label{eq41}%
\end{equation}
Therefore, by $C_{k}=\dfrac{a-b}{\pi k}\sin kc$,~\eqref{eq37},~\eqref{eq40},
and~\eqref{eq41}, we can calculate:
\begin{align*}
&  -C_{2n}+\frac{1}{\pi}\int_{0}^{\pi}Q^{2}(x)\cos
2nxdx+D(n,q)+2\operatorname{Re}D(n,p)\\
&  \qquad=-\frac{a-b}{2\pi n}\sin2nc-\frac{ab\cos2nc}{2n^{2}}+\frac{ab}%
{4n^{2}}+O\bigl(\frac{1}{n^{3}}\bigr).
\end{align*}
Thus, by~\eqref{eq30}, we have the following result:

\begin{theorem}
\label{t2} If $q(x)$ has the form~\eqref{eq9}, then the following asymptotic
formula holds:
\[
\lambda_{n}=n^{2}+\frac{b-a}{2\pi n}\sin2nc+\frac{ab(1-2\cos2nc)}{4n^{2}%
}+O\bigl(\frac{1}{n^{3}}\bigr).
\]

\end{theorem}

\section{Estimates for the Small Eigenvalues}

\label{Sec:3} For simplicity of reading, we first give the main ideas of the
proofs of the main results of this section. To give estimates for the small
Dirichlet eigenvalues, first, we prove (See Theorem~\ref{t3}) that Dirichlet
eigenvalues satisfy the equation
\begin{equation}
\lambda-n^{2}+C_{2n}-\sum\limits_{k=1}^{\infty}a_{k}(\lambda)=0, \label{eq42}%
\end{equation}
in the set $I_{n}:=[n^{2}-M,n^{2}+M]$, where $M=\max\{|a|,b\}$, and the
infinite series $a_{k}$ is defined by~\eqref{eq5}, if the following condition holds:

\begin{condition}
\label{c1} To consider the $n$th eigenvalue $\lambda_{n}$, for $n>1$, we
assume that the condition $M<(2n-1)/2$ holds. In case $n=1$, we assume that
the condition $M\leq1/2$ holds.
\end{condition}

Then, to use numerical methods, we take finite sums instead of the infinite
series in~\eqref{eq42}. To approximate the roots of equation~\eqref{eq42}, we
use the fixed point iteration. Using the Banach fixed point theorem, we prove
that the equation containing the finite sums has a unique solution in the
appropriate set $I_{n}$ (see Theorem~\ref{t4}). Moreover, we give error
analysis (see Theorem~\ref{t4} and Theorem~\ref{t5}) and present a numerical example.

Now, we state some preliminary facts. It is well known that the spectrum of
the operator $L(q)$ is discrete and for large enough $n$, there is one
eigenvalue (counting multiplicity) in the neighborhood of $n^{2}$. See the
basic and detailed classical results in~\cite{1, 2, 4, 5} and references
therein. The eigenvalues of the operator $L(0)$ are $n^{2}$, for
$n\in\mathbb{N}$, where $\mathbb{N}$ is the set of positive integers, and all
the eigenvalues of $L(0)$ are simple.

It is also known that~\cite{6}, $|\lambda_{n}-n^{2}|\leq M$, for $n\geq1$.
Therefore, we have
\begin{equation}
n^{2}-M\leq\lambda_{n}\leq n^{2}+M, \label{eq43}%
\end{equation}
for $n\geq1$. If $k\neq\pm n$, then $|\lambda_{n}-k^{2}|\geq|n^{2}%
-k^{2}|-M=|n-k||n+k|-M\geq(2n-1)-M$, for $n\geq1$ and under the assumptions
about $M$ given in Condition~\ref{c1}. In particular, if $n=1$, we have
$|\lambda_{1}|\leq1+M$ and $|\lambda_{1}-k^{2}|\geq||\lambda_{1}|-k^{2}%
|\geq4-|\lambda_{1}|\geq3-M$, for $|k|\geq2$. Besides, if $n\geq2$, we have
$|\lambda_{n}|\geq|\lambda_{2}|\geq4-M$ and $|\lambda_{n}-k^{2}|\geq
||\lambda_{2}|-k^{2}|\geq|\lambda_{2}|-1\geq3-M$, for $k\neq\pm n$.

We stress that, the iteration formula~\eqref{eq2} was used in~\cite{13} for
large eigenvalues to obtain asymptotic formulas. In this section, we find
conditions on potential~\eqref{eq9} for which the iteration
formula~\eqref{eq2} is also valid for the small eigenvalues, as $m$ tends to
infinity. We also note that, it is not easy to give such conditions, there are
many technical calculations.

We remind that, since the potential $q$ is of the form~\eqref{eq9}, we have
$C_{k}=\dfrac{a-b}{\pi k}\sin kc$ and $C_{-k}=C_{k}$, for $k=1,2,\ldots$. Now,
in order to give the main results, we prove the following lemma. Without loss
of generality, we assume that $\Psi_{n}(x)$ is a normalized eigenfunction
corresponding to the eigenvalue $\lambda_{n}$.

\begin{lemma}
\label{l4} If the assumptions about $M$ given in Condition~\ref{c1}
hold,\ then the statements

\textbf{(a)} $\lim_{m\rightarrow\infty}R_{m}(\lambda)=0,$ and \textbf{(b)}
$|(\Psi_{n},\sin nx)|>0$

are valid, where\ $R_{m}(\lambda)$ is defined by~\eqref{eq6}.

\begin{proof}
\textbf{ (a)} Since $||\Psi_{n}||=1$ and $||\sin kx||=\sqrt{\pi}/\sqrt{2}$, by
the Schwarz inequality, we have $|(q\Psi_{n},\sin kx)|\leq M\sqrt{\pi}%
/\sqrt{2}$. Considering the number of terms and the greatest summands of
$R_{2m}(\lambda)$ in absolute value, by~\eqref{eq43}, we obtain,
\[
|R_{2m}(\lambda_{1})|<\frac{4^{m}|C_{1}||C_{2}|^{2m}M\sqrt{\pi}/\sqrt{2}%
}{(2n-1-M)^{2m+1}}<\frac{(b-a)}{\sqrt{2\pi}}(\frac{2}{\pi})^{2m}.
\]
Therefore, $\lim_{m\rightarrow\infty}R_{m}(\lambda)=0$.

\textbf{(b)} Suppose the contrary, $(\Psi_{n},\sin nx)=0$. Since the system of
root functions $\{\sqrt{2}\sin kx/\sqrt{\pi}:k\in\mathbb{N}\}$ of $L(0)$ forms
an orthonormal basis for $L_{2}[0,\pi]$, we have the decomposition
\[
\dfrac{\pi}{2}\Psi_{n}=(\Psi_{n},\sin nx)\sin nx+\sum_{k\in\mathbb{N},k\neq
n}\left(  \Psi_{n},\sin kx\right)  \sin kx
\]
for the normalized eigenfunction $\Psi_{n}$ corresponding to the eigenvalue
$\lambda_{n}$ of $L(q)$. By Parseval's equality, we obtain
\begin{equation}
\sum_{k\in\mathbb{N},k\neq n}|\left(  \Psi_{n},\sin kx\right)  |^{2}=\frac
{\pi}{2} \label{eq44}%
\end{equation}
On the other hand, using the relation $(\lambda_{N}-n^{2})(\Psi_{N},\sin
nx)=(q\Psi_{N},\sin nx)$, the Bessel inequality and~\eqref{eq43}, we obtain
\begin{align*}
&  \sum_{k\in\mathbb{N},k\neq n}|\left(  \Psi_{n},\sin kx\right)  |^{2}%
=\sum_{k\in\mathbb{N},k\neq n}\frac{|(q\Psi_{n},\sin kx)|^{2}}{|\lambda
_{n}-k^{2}|^{2}}\\
&  \qquad\leq\frac{1}{(2n-1-M)^{2}}\sum_{k\in\mathbb{Z},k\neq\pm n}|(q\Psi
_{n},\sin kx)|^{2}<\frac{\pi(M)^{2}}{2(2n-1-M)^{2}}<\frac{\pi}{2}.
\end{align*}
which contradicts Parseval's equality~\eqref{eq44}.
\end{proof}
\end{lemma}

Before giving the main results, we calculate the number $-C_{2n}+A_{1,n}%
(n^{2})-B_{1,n}(n^{2})$ by~\eqref{eq14},~\eqref{eq37},~\eqref{eq40},
and~\eqref{eq41}:
\begin{align}
&  -C_{2n}+A_{1,n}(n^{2})-B_{1,n}(n^{2})=\frac{b-a}{2\pi n}\sin2nc+\frac
{ab(1-2\cos2nc)}{4n^{2}}+\frac{(b-a)(\cos2nc-1)(2a\sin2nc+ib)}{16n^{3}%
}\nonumber\\
&  \qquad+\frac{(b^{2}-a^{2})(2\sin2nc(2+\cos2nc)+i(\cos2nc-1))}{16\pi n^{3}%
}+\frac{(b-a)^{2}\cos2nc(\cos2nc-1)}{16\pi^{2}n^{4}}\nonumber\\
&  \qquad-\frac{-(b-a)[2(\sin2nc+\sin(2n-1)c-\sin(4n-1)c)+ib(1-e^{-i2nc}%
)(e^{i2nc}+e^{ic})]}{16n^{2}(2n-1)}\nonumber\\
&  \qquad+\frac{(b-a)^{2}(\cos4nc-\cos2nc)}{32\pi n^{4}}+\frac{(b^{2}%
-a^{2})(e^{-i2nc}-1)-(b-a)^{2}(e^{ic}+e^{i(2n-1)c}-2\cos(4n-1)c)}{16\pi
n^{2}(2n-1)^{2}}. \label{eq45}%
\end{align}

Now, we state the following main result.

\begin{theorem}
\label{t3} If the assumptions about $M$ given in Condition~\ref{c1} hold, then
$\lambda_{n}$ is an eigenvalue of $L(q)$ if and only if it is the root of
equation~\eqref{eq42} in the set $I_{n}:=[n^{2}-M,n^{2}+M]$.
\end{theorem}

\begin{proof}
\textbf{(a)} By Lemma~\ref{l4}, letting $m$ tend to infinity in
equation~\eqref{eq2}, we obtain
\[
\lambda_{n}-n^{2}+C_{2n}-\sum\limits_{k=1}^{\infty}a_{k}(\lambda_{n})=0.
\]
Now, we prove that the root of~\eqref{eq42} lying in the interval $I_{n}$ is
an eigenvalue of the operator $L$. To do this, we first give the following
relation:
\[
a_{1}(\lambda)=a_{1}(n^{2})+\sum_{\substack{n_{1}=-\infty\\n_{1}\neq
0,2n}}^{\infty}\frac{C_{n_{1}}(C_{n_{1}}-C_{2n+n_{1}})(\lambda-n^{2})}%
{n_{1}(2n+n_{1})[\lambda-(n+n_{1})^{2}]}%
\]
where $a_{1}(\lambda)$\ is defined in~\eqref{eq3}, $a_{1}(n^{2})=A_{1,n}%
(n^{2})-B_{1,n}(n^{2})$, $A_{1,n}(n^{2})$\ and $B_{1,n}(n^{2})$\ are given in~\eqref{eq11}.

The equation
\[
F_{n}(\lambda):=\lambda-n^{2}+C_{2n}-A_{1,n}(n^{2})+B_{1,n}(n^{2})
\]
has one root in the interval $I_{n}=[n^{2}-M,n^{2}+M]$ and
\[
|F_{n}(\lambda)|\geq|\lambda-n^{2}|-|C_{2n}-A_{1,n}(n^{2})+B_{1,n}(n^{2})|.
\]
Estimating $|C_{2n}-A_{1,n}(n^{2})+B_{1,n}(n^{2})|$ by~\eqref{eq45}, we have\
\[
|F_{n}(\lambda)|>0.16,
\]
for all $\lambda$ from the boundary of $I_{n}$, for $n=1$. Now we define
\[
G_{n}(\lambda):=\lambda-n^{2}+C_{2n}-\sum\limits_{k=1}^{\infty}a_{k}%
(\lambda_{n}).
\]
Estimating the summands of $|a_{1}(\lambda)-a_{1}(n^{2})|$\ and $|a_{k}%
(\lambda)|$ by considering the greatest summands of them,\ for $n=1$, we
obtain
\[
|a_{1}(\lambda)-a_{1}(n^{2})|<\frac{16}{45\pi^{2}},\qquad|a_{k}(\lambda
)|<\frac{\sqrt{2}^{k}}{\pi^{k+1}},
\]
for $k\geq2$. Therefore, it follows by the geometric series formula that
\[
|a_{1}(\lambda)-a_{1}(n^{2})|+\sum\limits_{k=2}^{\infty}|a_{k}(\lambda
)|<\frac{16}{45\pi^{2}}+\frac{\pi+1}{\pi^{2}(\pi-1)}<0.1534
\]
for $n=1$. Hence
\begin{align*}
&  |G_{n}(\lambda)-F_{n}(\lambda)|=\biggl|a_{1}(\lambda)-a_{1}(n^{2}%
)+\sum\limits_{k=2}^{\infty}a_{k}(\lambda)\biggr|\\
&  \qquad\leq|a_{1}(\lambda)-a_{1}(n^{2})|+\sum\limits_{k=2}^{\infty}%
|a_{k}(\lambda)|<\frac{16}{45\pi^{2}}+\frac{\pi+1}{\pi^{2}(\pi-1)}<0.1534
\end{align*}
for all $\lambda$ from the boundary of $I_{n}$, for $n=1$. Similar estimations
can be obtained also for $n\geq2$. Therefore $|G_{n}(\lambda)-F_{n}%
(\lambda)|<|F_{n}(\lambda)|$ holds for all $\lambda$ from the boundary of
$I_{n}$ and by Rouche's theorem, $G_{n}(\lambda)$ has one root in the set
$I_{n}$, for $n\geq1$. Hence, $L(q)$ has one eigenvalue (counting
multiplicity) in $I_{n}$, which is the root of~\eqref{eq42}. On the other
hand, equation~\eqref{eq42} has exactly one root (counting multiplicity) in
$I_{n}$. Thus, $\lambda\in I_{n}$ is an eigenvalue of $L(q)$ if and only if,
it is the root of~\eqref{eq42}.
\end{proof}

Now, we can use numerical methods by taking finite sums instead of the
infinite series in~\eqref{eq42} and write
\[
\lambda-n^{2}+C_{2n}-\sum\limits_{k=1}^{s}a_{r,k,n}(\lambda)=0,
\]
where $a_{r,k,n}(\lambda)$ is obtained from~\eqref{eq5} by taking finite sums
instead of the infinite series, namely,
\[
a_{r,k,n}(\lambda):=\sum_{n_{1},n_{2},...,n_{k}=-r}^{r}\frac{C_{n_{1}}%
C_{n_{2}}\cdots C_{n_{k}}(C_{n_{1}+n_{2}+\cdots+n_{k}}-C_{n_{1}+n_{2}%
+\cdots+n_{k}+2n})}{[\lambda-(n+n_{1})^{2}]\cdots\lbrack\lambda-(n+n_{1}%
+\cdots+n_{k})^{2}]}.
\]
Define the functions
\begin{equation}
K_{n}(\lambda):=\lambda-n^{2}-g_{n}(\lambda),\qquad g_{n}(\lambda
)=-C_{2n}+\sum\limits_{k=1}^{s}a_{r,k,n}(\lambda). \label{eq46}%
\end{equation}
Then,
\begin{equation}
\lambda=n^{2}+g_{n}(\lambda), \label{eq47}%
\end{equation}
for $n\geq1$.

Now we state another main result.

\begin{theorem}
\label{t4} Suppose that the assumptions about $M$ given in Condition~\ref{c1}
hold. Then for all $x$ and $y$ from the interval $I_{n}=[n^{2}-M,n^{2}+M]$,
the relations
\begin{gather}
|g_{n}(x)-g_{n}(y)|\leq L_{n}|x-y|,\label{eq48}\\
L_{n}=\frac{9(b-a)^{2}}{4\pi(2n-1-M)[4\pi(2n-1-M)-3(b-a)]}\leq\frac{9}%
{2\pi(2\pi-3)}<1,\nonumber
\end{gather}
hold and equation~\eqref{eq47} has a unique solution $\rho_{n}$ in $I_{n}$.
Moreover
\begin{align}
&  |\lambda_{n}-\rho_{n}|<\frac{(b-a)^{s+2}}{2M\sqrt{2}^{s}\pi^{s+1}%
(2n-1-M)^{s}[\sqrt{2}\pi(2n-1-M)-(b-a)](1-L_{n})}\nonumber\\
&  \qquad+\frac{8(b-a)^{2}}{\pi^{2}(r+1)^{2}[(r+1)|r+1-2n|-M](1-L_{n})}.
\label{eq49}%
\end{align}

\end{theorem}

\begin{proof}
First, we prove~\eqref{eq48} by using the mean-value theorem. To do this, we
estimate $|g_{n}^{\prime}(\lambda)|=|\dfrac{d}{d\lambda}g_{n}(\lambda)|$.
By~\eqref{eq46}, we have
\[
|g_{n}^{\prime}(\lambda)|=\biggl|\sum\limits_{k=1}^{s}\frac{d}{d\lambda
}a_{r,k,n}(\lambda)\biggr|\leq\sum\limits_{k=1}^{s}\bigl|\frac{d}{d\lambda
}a_{r,k,n}(\lambda)\bigr|.
\]
By estimating the summands of the term $|\dfrac{d}{d\lambda}a_{r,k,n}%
(\lambda)|$, we write
\begin{align*}
&  \biggl|\frac{d}{d\lambda}(a_{r,k,n}(\lambda))\biggr|=\biggl|\sum
_{n_{1},n_{2},...,n_{k}=-r}^{r}\frac{d}{d\lambda}\frac{C_{n_{1}}C_{n_{2}%
}\cdots C_{n_{k}}(C_{n_{1}+n_{2}+\cdots+n_{k}}-C_{n_{1}+n_{2}+\cdots+n_{k}%
+2n})}{[\lambda-(n+n_{1})^{2}]\cdots\lbrack\lambda-(n+n_{1}+\cdots+n_{k}%
)^{2}]}\biggr|\\
&  \qquad<\frac{3^{k+1}(b-a)^{k+1}}{4^{k+1}\pi^{k+1}(2n-1-M)^{k+1}}%
\leq\bigl|\frac{3}{2\pi}\bigr|^{k+1},
\end{align*}
for $k\geq1$. Thus, by the geometric series formula, we obtain
\[
\sum\limits_{k=1}^{r}\bigl|\frac{d}{d\lambda}a_{r,k,n}(\lambda)\bigr|<\frac
{9}{2\pi(2\pi-3)}.
\]
Hence,
\begin{equation}
|g_{n}^{\prime}(\lambda)|<\frac{9(b-a)^{2}}{4\pi(2n-1-M)[4\pi(2n-1-M)-3(b-a)]}%
=L_{n}\leq\frac{9}{2\pi(2\pi-3)}<1.\label{eq50}%
\end{equation}
Since the inequality
\begin{equation}
|g_{n}^{\prime}(\lambda)|<L_{n}<1\label{eq51}%
\end{equation}
holds for all $x,y\in I_{n}$,~\eqref{eq48} holds by the mean value theorem and
equation~\eqref{eq47} has a unique solution $\rho_{n}$ in $I_{n}$, by the
contraction mapping theorem. Now let us prove~\eqref{eq49}. By~\eqref{eq46},
we have $K_{n}(x)=x-n^{2}-g_{n}(x)$ and by the definition of $\rho_{n}$, we
write $K_{n}(\rho_{n})=0$. Therefore by~\eqref{eq42} and~\eqref{eq47}, we
obtain
\begin{align}
&  |K_{n}(\lambda_{n})-K_{n}(\rho_{n})|=|K_{n}(\lambda_{n})|\nonumber\\
&  \qquad=\biggl|\lambda_{n}-n^{2}+C_{2n}-\sum\limits_{k=1}^{s}a_{r,k,n}%
(\lambda_{n})\biggr|\nonumber\\
&  \qquad\leq\biggl|\sum\limits_{k=1}^{\infty}a_{k}(\lambda_{n})-\sum
\limits_{k=1}^{s}a_{r,k,n}(\lambda_{n})\biggr|.\label{eq52}%
\end{align}
For the estimation of the right-hand side of~\eqref{eq52}, we have
\begin{align*}
&  \biggl|\sum\limits_{k=1}^{\infty}a_{k}(\lambda_{n})-\sum\limits_{k=1}%
^{s}a_{r,k,n}(\lambda_{n})\biggr|\\
&  \qquad\leq\biggl|\sum\limits_{k=1}^{\infty}a_{k}(\lambda_{n})-\sum
\limits_{k=1}^{s}a_{k}(\lambda_{n})\biggr|+\biggl|\sum\limits_{k=1}^{s}%
a_{k}(\lambda_{n})-\sum\limits_{k=1}^{s}a_{r,k,n}(\lambda_{n})\biggr|\\
&  \qquad\leq\sum\limits_{k=s+1}^{\infty}|a_{k}(\lambda_{n})|+\sum
\limits_{k=1}^{s}|a_{k}(\lambda_{n})-a_{r,k,n}(\lambda_{n})|.
\end{align*}
Using the estimations for $\sum\limits_{k=s+1}^{\infty}|a_{k}(\lambda_{n}%
)|$\ and $\sum\limits_{k=1}^{s}|a_{k}(\lambda_{n})-a_{r,k,n}(\lambda_{n})|$,
by considering the greatest summands of them, we obtain
\begin{align}
&  \biggl|\sum\limits_{k=1}^{\infty}a_{k}(\lambda_{n})-\sum\limits_{k=1}%
^{s}a_{r,k,n}(\lambda_{n})\biggr|<\sum\limits_{k=s+1}^{\infty}\frac
{(b-a)^{k+1}}{2M\sqrt{2}^{k}\pi^{k+1}(2n-1-M)^{k}}+\frac{8(b-a)^{2}}{\pi
^{2}(r+1)^{2}[(r+1)|r+1-2n|-M]}\nonumber\\
&  \qquad=\frac{(b-a)^{s+2}}{2M\sqrt{2}^{s}\pi^{s+1}(2n-1-M)^{s}[\sqrt{2}%
\pi(2n-1-M)-(b-a)]}+\frac{8(b-a)^{2}}{\pi^{2}(r+1)^{2}[(r+1)|r+1-2n|-M]}%
.\label{eq53}%
\end{align}
Thus, by~\eqref{eq52} and~\eqref{eq53}, we have
\begin{align}
&  |K_{n}(\lambda_{n})-K_{n}(\rho_{n})|<\frac{(b-a)^{s+2}}{2M\sqrt{2}^{s}%
\pi^{s+1}(2n-1-M)^{s}[\sqrt{2}\pi(2n-1-M)-(b-a)]}\nonumber\\
&  \qquad+\frac{8(b-a)^{2}}{\pi^{2}(r+1)^{2}[(r+1)|r+1-2n|-M]}.\label{eq54}%
\end{align}
To apply the mean value theorem, we estimate $|K_{n}^{\prime}(\lambda)|$:
\begin{equation}
|K_{n}^{\prime}(\lambda)|=|1-g_{n}^{\prime}(\lambda)|\geq|1-|g_{n}^{\prime
}(\lambda)||\geq1-L_{n}.\label{eq55}%
\end{equation}
By the mean value formula,~\eqref{eq50},~\eqref{eq51},~\eqref{eq54}
and~\eqref{eq55}, we obtain
\[
|K_{n}(\lambda_{n})-K_{n}(\rho_{n})|=|K_{n}^{\prime}(\xi)||\lambda_{n}%
-\rho_{n}|,\qquad\xi\in\lbrack n^{2}-M,n^{2}+M]
\]
and
\begin{align*}
&  |\lambda_{n}-\rho_{n}|=\frac{|K_{n}(\lambda_{n})-K_{n}(\rho_{n})|}%
{|K_{n}^{\prime}(\xi)|}<\frac{(b-a)^{s+2}}{2M\sqrt{2}^{s}\pi^{s+1}%
(2n-1-M)^{s}[\sqrt{2}\pi(2n-1-M)-(b-a)](1-L_{n})}\\
&  \qquad+\frac{8(b-a)^{2}}{\pi^{2}(r+1)^{2}[(r+1)|r+1-2n|-M](1-L_{n})},
\end{align*}
which completes the proof.
\end{proof}

Now let us approximate $\rho_{n}$ by the fixed point iterations:
\begin{equation}
x_{n,i+1}=n^{2}+g_{n}(x_{n,i}), \label{eq56}%
\end{equation}
where $g_{n}(x)$ is defined in~\eqref{eq46}. Since
\begin{equation}
|g_{n}(\lambda_{n})|=\biggl|-C_{2n}+\sum\limits_{k=1}^{s}a_{r,k,n}(\lambda
_{n})\biggr|\leq|C_{2n}|+\sum\limits_{k=1}^{s}|a_{r,k,n}(\lambda_{n})|,
\label{eq57}%
\end{equation}
first let us estimate $|a_{r,k,n}(\lambda)|$:
\begin{align}
&  \sum\limits_{k=1}^{s}|a_{r,k,n}(\lambda_{n})|<\frac{(b-a)^{2}}{M\pi
^{2}(2n-1-M)}+\sum\limits_{k=2}^{r}\frac{(b-a)^{k+1}}{2M\sqrt{2}^{k}\pi
^{k+1}(2n-1-M)^{k}}\nonumber\\
&  \qquad<\frac{4}{\pi^{2}}+\sum\limits_{k=2}^{r}\frac{\sqrt{2}^{k}}{\pi
^{k+1}}<\frac{4}{\pi^{2}}+\frac{2}{\pi^{2}(\pi-\sqrt{2})}. \label{eq58}%
\end{align}
Therefore,
\[
|g_{n}(\lambda_{n})|\leq|C_{2n}|+\sum\limits_{k=1}^{s}|a_{r,k,n}(\lambda
_{n})|<\frac{(b-a)}{2\pi n}+\frac{4}{\pi^{2}}+\frac{2}{\pi^{2}(\pi-\sqrt{2}%
)}.
\]
On the other hand, writing $(2n-1)$ instead of $(2n-1)-M$ in~\eqref{eq57}
and~\eqref{eq58}, we obtain
\begin{align}
&  |g_{n}(n^{2})|\leq|C_{2n}|+\sum\limits_{k=1}^{s}|a_{r,k,n}(n^{2}%
)|<\frac{(b-a)}{2\pi n}+\frac{(b-a)^{2}}{M\pi^{2}(2n-1)}+\sum\limits_{k=2}%
^{r}\frac{(b-a)^{k+1}}{2M\sqrt{2}^{k}\pi^{k+1}(2n-1)^{k}}\nonumber\\
&  \qquad<\frac{(b-a)}{2\pi n}+\frac{(b-a)^{2}}{M\pi^{2}(2n-1)}+\frac
{(b-a)^{3}}{2\sqrt{2}M\pi^{2}(2n-1)[\sqrt{2}\pi(2n-1)-(b-a)]}, \label{eq59}%
\end{align}
since $|n^{2}-k^{2}|\geq2n-1$, for $n=1,2,\ldots$. Now we state the following result.

\begin{theorem}
\label{t5} If the assumptions about $M$ given in Condition~\ref{c1} hold, then
the following estimation holds for the sequence $\{x_{n,i}\}$ defined
by~\eqref{eq56}:
\begin{align}
&  |x_{n,i}-\rho_{n}|<(L_{n})^{i}\biggl(\frac{(b-a)}{2\pi n(1-L_{n})}%
+\frac{(b-a)^{2}}{M\pi^{2}(2n-1)(1-L_{n})}\nonumber\\
&  \qquad+\frac{(b-a)^{3}}{2\sqrt{2}M\pi^{2}(2n-1)[\sqrt{2}\pi
(2n-1)-(b-a)](1-L_{n})}\biggr), \label{eq60}%
\end{align}
for $i=1,2,3,\ldots$, where $L_{n}$ is defined in~\eqref{eq48} in
Theorem~\ref{t4}.
\end{theorem}

\begin{proof}
Without loss of generality, we can take $x_{n,0}=n^{2}$.
By~\eqref{eq47},~\eqref{eq48}, and~\eqref{eq56}, we have
\begin{align}
&  |x_{n,i}-\rho_{n}|=|n^{2}+g_{n}(x_{n,i-1})-(n^{2}+g_{n}(\rho_{n}%
))|\nonumber\\
&  \qquad=|g_{n}(x_{n,i-1})-g_{n}(\rho_{n})|<L_{n}|x_{n,i-1}-\rho_{n}%
|<(L_{n})^{i}|x_{n,0}-\rho_{n}|. \label{eq61}%
\end{align}
Therefore, it is enough to estimate $|x_{n,0}-\rho_{n}|$. By definitions of
$\rho_{n}$ and $x_{n,0}$, we obtain
\[
\rho_{n}-x_{n,0}=g_{n}(\rho_{n})+n^{2}-x_{n,0}=g_{n}(\rho_{n})-g_{n}%
(x_{n,0})+g_{n}(n^{2})
\]
and by the mean value theorem, there exists $x\in\lbrack n^{2}-M,n^{2}+M]$
such that
\[
g_{n}(\rho_{n})-g_{n}(x_{n,0})=g_{n}^{\prime}(x)(\rho_{n}-x_{n,0}).
\]
The last two equalities imply that
\[
(\rho_{n}-x_{n,0})(1-g_{n}^{\prime}(x))=g_{n}(n^{2}).
\]
Hence by~\eqref{eq51},~\eqref{eq59}, and~\eqref{eq61}, we have
\begin{align*}
&  |\rho_{n}-x_{n,0}|\leq\frac{|g_{n}(n^{2}))|}{1-L_{n}}<\frac{(b-a)}{2\pi
n(1-L_{n})}+\frac{(b-a)^{2}}{M\pi^{2}(2n-1)(1-L_{n})}\\
&  \qquad+\frac{(b-a)^{3}}{2\sqrt{2}M\pi^{2}(2n-1)[\sqrt{2}\pi
(2n-1)-(b-a)](1-L_{n})}%
\end{align*}
and
\begin{align*}
&  |x_{n,i}-\rho_{n}|<(L_{n})^{i}\biggl(\frac{(b-a)}{2\pi n(1-L_{n})}%
+\frac{(b-a)^{2}}{M\pi^{2}(2n-1)(1-L_{n})}\\
&  \qquad+\frac{(b-a)^{3}}{2\sqrt{2}M\pi^{2}(2n-1)[\sqrt{2}\pi
(2n-1)-(b-a)](1-L_{n})}\biggr).
\end{align*}
The theorem is proved.
\end{proof}

Thus by~\eqref{eq49} and \eqref{eq60}, we have the approximation $x_{n,i}$ for
the Dirichlet eigenvalue $\lambda_{n}$ with the error
\begin{align*}
&  |\lambda_{n}-x_{n,i}|<\frac{(b-a)^{s+2}}{2M\sqrt{2}^{s}\pi^{s+1}%
(2n-1-M)^{s}[\sqrt{2}\pi(2n-1-M)-(b-a)](1-L_{n})}\\
&  \qquad+\frac{8(b-a)^{2}}{\pi^{2}(r+1)^{2}[(r+1)|r+1-2n|-M](1-L_{n})}%
+(L_{n})^{i}\biggl(\frac{(b-a)}{2\pi n(1-L_{n})}\\
&  \qquad+\frac{(b-a)^{2}}{M\pi^{2}(2n-1)(1-L_{n})}+\frac{(b-a)^{3}}{2\sqrt
{2}M\pi^{2}(2n-1)[\sqrt{2}\pi(2n-1)-(b-a)](1-L_{n})}\biggr).
\end{align*}
By this error formula, it is clear that the error gets smaller as $s$ and $r$\ increase.

Now, we present a numerical example:

\begin{example}
For $a=-1/2$, $b=1/2$, and $c=\pi/2$, we have the following approximations for
the first Dirichlet eigenvalues $\lambda_{1}$, $\lambda_{2}$, $\lambda_{3}$,
$\lambda_{4}$, $\lambda_{5}$ and $\lambda_{6}$. In our calculations, we take
$s=r=5$.
\begin{align*}
\lambda_{1}  &  =0.984205232093,\quad\lambda_{2}=4.003110832419,\quad
\lambda_{3}=9.013023482675\\
\lambda_{4}  &  =16.018415152245,\quad\lambda_{5}=25.010308870985,\quad
\lambda_{6}=36.010794654577.
\end{align*}
Usually it takes $8-10$ iterations with the tolerance $1e-18$\ by the fixed
point iteration method, even if we choose an initial value that is not too
close to the exact value, which means that convergence is quite fast.
\end{example}

\end{document}